\documentclass[a4paper,12pt]{article}
\usepackage{amssymb,amsmath,amsthm}
\usepackage{graphicx}

\usepackage{amsfonts}
\usepackage{latexsym}
\usepackage{color}
\usepackage[english]{babel}
\usepackage{hyperref}

\usepackage{hyperref}
\usepackage{verbatim} 
\usepackage{xcolor}
\usepackage{graphicx} % Allows including images
\usepackage{subcaption}
\usepackage{caption}
\usepackage{mathrsfs}
\numberwithin{equation}{section}

\newtheorem{lemma}{Lemma}[section]
\newtheorem{theorem}[lemma]{Theorem}
\newtheorem{proposition}[lemma]{Proposition}
\newtheorem{corollary}[lemma]{Corollary}
\newtheorem{remark}[lemma]{Remark}

\newtheorem{definition}[lemma]{Definition}

\def\im{{\rm i}}

\def\cU{{\mathcal{U}}}

\def\cR{{\mathcal{R}}}
\def\cP{{\mathcal{P}}}

\def\cT{{\mathcal{T}}}

\def\resto{{\mathcal{R}}}

\def\poi#1#2{\left\{#1;#2\right\}}
\newcommand{\R}{\mathbb R}

\newcommand{\Z}{\mathbb Z}

\newcommand{\N}{\mathbb N}
\newcommand{\T}{\mathbb T}

\def\norma#1{\Vert#1\Vert}
\def\sy#1{\cP^{#1,\delta}}

\newcommand{\ep}{\varepsilon}

\newcommand{\res}{{(\rm{res})}}
\newcommand{\nr}{{(\rm{nr})}}

\newcommand{\dist}{\textrm{dist}}
\newcommand{\p}{p} %{\widetilde{p}}

\newcommand{\pp}{\overline{p}}
\newcommand{\qq}{\overline{q}}

\def\ac{{\mathcal A}\kern-.7pt\ell\kern-.9pt\mathcal{S}}

\begin{document}

\title{A simple proof for a $C^\infty$ Nekhoroshev theorem } \author{
  D. Bambusi$^{*}$, B. Langella\footnote{Dipartimento di Matematica,
    Universit\`a degli Studi di Milano, Via Saldini 50, I-20133
    Milano. \newline \textit{Email: } \texttt{dario.bambusi@unimi.it,
      beatrice.langella@unimi.it}}}
	
	\maketitle
	
	\begin{abstract}
We prove a $C^\infty$ version of the Nekhoroshev's estimate on the
stability times of the actions in close to integrable Hamiltonian
systems. The proof we give is a variant of the original Nekhoroshev's
proof and it consists in first conjugating, globally in the phase space,
and up to a small remainder, the system to a normal form.  Then we
perform the geometric part of the proof in the normalized
variables. As a result, we obtain a proof which is simpler than
the usual ones.
	\end{abstract}

\section{Introduction}

In this paper we prove a $C^{\infty}$ version of Nekhoroshev's Theorem
for the stability times in a close to integrable Hamiltonian
system. The proof we give is much simpler than the usual one and leads
to an intermediate result that we think could have some interest in
itself (see Theorem \ref{main2}).

To be definite and in order to avoid as much as possible technical
complications, we study here
a system of the form
\begin{equation}
  \label{H1}
\begin{gathered}
  H(p, q) = H_0(p) + \ep  V(p,q)\,,\\ H_0(p) =
\sum_{j=1}^{d}\frac{p_j^2}{2}\,,
\end{gathered}
\end{equation}
with $V\in C^\infty(\T^d\times\R^d)$. However, 
our technique is applicable also to the general case of 
perturbations of steep integrable systems, and with a perturbation
which is not globally bounded in the momenta $p$.

The result we get is the following version of Nekhoroshev's Theorem.
\begin{theorem}
  \label{nek}Assume that $V\in C^{\infty}(\T^d\times\R^d)$ is globally bounded and
fix a positive $b<\frac{1}{2}$. Then,  for any positive $M$, there exists
$C_M,\ep _M$ such that, if $0<\ep <\ep _M$ then, corresponding
to any initial datum, one has
\begin{align}
  \label{sti_az}
  \left|p(t)-p(0)\right|\leq C_M\ep ^b\ ,
  \\
  \label{times}
  |t|\leq
\frac{C_M }{\ep ^M}\ .
\end{align}
\end{theorem}
We recall that in the analytic (or Gevrey) case the time of stability,
cf. eq. \eqref{times},
is known to be exponentially long.

Theorem \ref{nek} is not new, for example it is a direct consequence
of Theorem 2.1 of \cite{Bou10}\footnote{Actually such a theorem is
  also stated in \cite{MS04}, where reference is made to a slightly
  different statement present in \cite{MS02}}. However, as far as we
know our proof is new and the value of the exponent $b$ that we get is
better than those present in literature\footnote{In the analytic case
  any $b<1/2$ is allowed at the price of worsening the estimate on
  the times, see for example \cite{poschel}}.

Essentially two methods of proof of Nekhoroshev's theorem are known:
the original one \cite{Nek1,Nek2,BGG85,BG86,poschel,GZ92, chierchia},
and Lochak's one \cite{Loc,Loc2} (see also \cite{MS02, MS04} and
\cite{BG93,Bam99} for infinite dimensional generalizations). Lochak's
method was also extended to the steep case by using also ideas from
the original proof by Nekhoroshev \cite{Nie04}, see also
\cite{Nie07,Bou11,BN12}.

Our proof is a variant of Nekhoroshev's original one which consists of
two steps: the analytic part and the geometric part. Classically, in
the analytic part one shows that in a region of the phase space where
only some resonances are present one can conjugate, up to a small
remainder, the system to a system in resonant normal form. In the
geometrical part one collects all the information and shows that, if
the regions are suitably constructed, then for any initial datum there
exists a region in which it remains for the considered times, and this
leads to Nekhoroshev's estimate.

In the classical approach the analysis of the geometrical part is
slightly complicated by the fact that it has to be performed in the
original coordinates, so that one has to take into account the effects
of the coordinate transformation used to conjugate the system to
normal form. The novelty of the present paper is that we use a
canonical transformation which is globally defined and globally
conjugates the system to a normal form which is different in each
region of the phase space, depending on the resonances which are
present in each region (a similar technique has been used for the
first time in a probabilistic context in \cite{olandesi,BMT19}).  This
is obtained by splitting each Fourier coefficient of $V$, namely $
\hat V_k(p) $, into a part localized in the region $|\omega(p)\cdot
k|<\ep ^\delta$ (with a suitable $\delta$) and a part localized in
the nonresonant region. The part localized in the nonresonant region
is then removed through the normalizing canonical
transformation. Technically the localization is obtained simply by
multiplying by a smooth cutoff function.

Then the geometrical part consists in making a decomposition of the
phase space in regions which are invariant for the dynamics of the
normalized system. This leads in particular to the conclusion that, in
the dynamics of the normalized system, estimate \eqref{sti_az} is
valid for all times. This is the content of the following theorem,
which, as far as we know, is new.
\begin{theorem}
  \label{main2}
Fix a positive $b<\frac{1}{2}$, then,  for any positive $M$, there exists
$C_M,\ep _M$ such that, if $0<\ep <\ep _M$ then there
exists a canonical transformation $(p,q)=\cT(\tilde p,\tilde q)$ and a
(normal form)
Hamiltonian $H_Z$, with
the following properties
\begin{itemize}
\item[1.] {$|p-\tilde p|\leq C_M\ep ^{b}$}
%$|p-\tilde p|\leq C_M\ep ^{b+\frac{1}{2}}$
\item[2.] $\norma{H\circ\cT-H_Z}_{C^{2}(\R^d\times\T^d)}\leq C_M\ep ^M$,
\item[3.] Along the solutions of the Hamilton equations of $H_Z$, one
  has
  $$ |{\tilde p(t)-\tilde p(0)}|\leq C_M\ep ^b\ ,\quad \forall
  t\in\R\ .
  $$ 
\end{itemize}
  \end{theorem}

When adding the remainder, one gets the limitation \eqref{times} on the
times.

Finally we remark that Lochak's proof applies to system \eqref{H}, but
we think that our approach to the geometric part of the proof is the
main interest of the present paper, since it is suitable for
generalizations to the steep case.

\vskip15truept

This paper originates from our research on the spectrum of Sturm
Liouville operators in general tori, which lead to a quantum version
of Nekhoroshev's theorem \cite{BLM19,BLM20}. When we were still lost
on how to construct a quantum analogue of the geometric part, we had
several very enlightening discussions with Antonio Giorgilli on the
classical Nekhoroshev's theorem. At the end we realized that the
quantum method we constructed had a classical counterpart which is the
content of the present paper. It is a pleasure to dedicate this paper
to Antonio Giorgilli in the occasion of his 70th birthday.

One of us, Dario Bambusi, would like to thank Antonio who introduced
him to science and in particular to the study of Hamiltonian dynamics:
his presence has always been fundamental and I would be a different
person if I had not met him. Thank you!

\noindent{\it Acknowledgements} We thank Laurent Niederman and David
Sauzin for pointing to our attention the relevant references on
Nekhoroshev Theorem in the smooth class. We acknowledge the support of
GNFM.

\section{Analytical part}

\subsection{Preliminaries and statement}

In this subsection we present the tools we will use in order to deal
with the $C^\infty$ context.

Having {fixed a parameter $0<\delta<\frac{1}{2}$}  and an interval 
$\cU=[0,\ep_0)$ with some positive $\ep_0$, we
give the following definition.

\begin{definition}
A family of functions $\left\{f_{\ep }\right\}_{\ep \in\cU}$,   $ f_\ep \in C^{\infty}(\R^d \times \T^d)$ will be said to be a
symbol of order $m$ if for all $\alpha, \beta \in \N^d$ there exists a
positive constant $C_{\alpha, \beta}$ such that
	\begin{equation} \label{ord m}
	\sup_{\ep \in {\cal U}} \sup_{p \in \R^d, q \in \T^d}
        \left|\frac{\partial_p^{\alpha} \partial_q^\beta f_\ep (p,
          q)}{\ep^{m}}\ep^{|\alpha|\delta }\right| \leq C_{\alpha, \beta}\,.
	\end{equation}
In this case we will denote  $f_\ep  \in {\cal P}^{m, \delta}$. We
will often omit the index $\ep $. 
\end{definition}
	\begin{remark} \label{semin vere}
It is immediate to see that $f \in {\cal P}^{m, \delta}$ if and only if for all
	 integers $N_1$ and $N_2$  there exists a positive constant $C^m_{N_1, N_2}$ such that
	\begin{equation} \label{def seminorme}
	\sup_{\ep \in {\cal U}} \sup_{\begin{subarray}{c}
		p \in \R^d,\ k \in \Z^d,\\ \alpha \in \N^d,\ |\alpha| = N_1
		\end{subarray}}  \left|\partial_p^{\alpha} \hat{f}_k(\ep, p)\right| |k|^{N_2} \ep^{-(m -|\alpha|\delta)}  \leq C^m_{N_1, N_2}\,,
	\end{equation}
	where 
	$$
	\hat{f}_k(\ep, p) =\frac{1}{(2 \pi)^d} \int_{\T^d} f_\ep (p, q) e^{-\im k \cdot q}, \quad \ep \in {\cal U}, \quad p \in \R^d\,
	$$
are the Fourier coefficients of $f$.
        \end{remark}
        
        \begin{remark}
          \label{fresch}
The space $\sy m$ endowed by the family of seminorms given by the
constants $C^m_{N_1,N_2}=C^m_{N_1,N_2}(f)$ of equation \eqref{def
  seminorme} is a Fr\'echet space. 
        \end{remark}

        \begin{remark}
	A direct computation shows that, if $f \in \sy{m_1}$ and $g
        \in \sy{m_2},$  then
	\begin{enumerate}
	\item $f + g \in \sy{\min \{ m_1, m_2\}}$ 
	\item  $f g \in \sy{m_1 + m_2}$
	\item the Poisson bracket, $\{ f, g\} \in \sy{m_1 + m_2
          -\delta}\,.$
	\end{enumerate}
	\end{remark}

        In the following, given a $C^\infty$ function $g$, we will denote by
$X_g$ its Hamiltonian vector field and by $\Phi_g^t$ the flow it
generates (which in our framework will always be globally defined). 

In order to state the analytic Lemma, we start by defining  what we
mean by normal form of order $N$. From now of{ we fix the number $N$
  controlling the number of steps in the normal form procedure.} 

Furthermore, we
will denote $$
a:=1-2\delta\ ;
$$
 we fix a positive (small) $0<\beta<1$ and we define 
\begin{equation}
  \label{kappa}
K=K(\ep):=\left[\frac{1}{\ep^\beta}\right]+1
\end{equation}
with the square bracket denoting the integer part. Eventually we will
link $ \beta$, $\delta$, $b$, $M$ and $N$.
\begin{definition} \label{normal form}
	A family of functions ${Z_{\ep}:\R^d \times \T^d \rightarrow
          \R}$ will be said to be in normal form if $${Z_{\ep}(p, q) =
          \sum_{ |k| \leq K} \hat{Z}_k(\ep, p) e^{\im k \cdot q} }$$
        with
	\begin{equation}
		\hat{Z}_k(\ep, p)  \neq 0 \Rightarrow |p \cdot k |
                \leq \ep^\delta\, , \quad \forall k \in \Z^d
                \backslash \{0\}\ ,
        	\end{equation}
Namely the $k$-th Fourier coefficient is supported in the resonant
region $\left|p\cdot k\right|\leq \ep^\delta$. 
\end{definition}

        \begin{lemma}
          \label{nor.for}(Normal Form Lemma) Consider the system
          \begin{equation}
            \label{H}
H_\ep:=H_0(p)+P_\ep(p,q)
            \end{equation}
with $H_0$ as in \eqref{H1} and $P_\ep\in \sy1$, then there exists a
canonical transformation ${{\cal T}}$ such that
          \begin{equation}
            \label{h.noma}
H_\ep\circ {{\cal T}}=H_0+\sum_{j=1}^N Z_j+\cR^{(N)}\ ,
          \end{equation}
          with $Z_j\in\sy{1+a(j-1)}$ in normal form and $\cR^{(N)}$
          s.t.
          \begin{align}
            \label{resti}
\sup_{\ep \in {\cal U}} \sup_{p \in \R^d, q \in \T^d}
        \left|\frac{\partial_p^{\alpha} \partial_q^\beta \cR^{(N)}(p,
        q)}{\ep^{1+Na}}\right|\ep^{|\alpha|\delta} \leq C_{\alpha,
          \beta}\, ,
        \\
                    \label{riresti}
\forall\alpha,\beta\in\N^d\times\N^d        \quad \text{with}\ |\alpha|+|\beta|\leq 1\ .
\end{align}
        Furthermore, given a symbol $f\in\sy m$,
        define 
$\cR_f:=f\circ {{\cal T}}-f $, then one has
       \begin{equation} \label{restir}
	 \sup_{p \in \R^d, q \in \T^d}
        \left|\cR_f(\ep, p,
        q)\right| \leq C \ep ^{m+1-2\delta}\ . 
       \end{equation}
       In the case $f=p_j$, $j=1,...,d$, one has 
  \begin{equation} \label{restip}
	 \sup_{p \in \R^d, q \in \T^d}
        \left|\cR_{p_j}(\ep, p,
        q)\right| \leq C \ep ^{1-\delta}\ . 
  \end{equation}
        \end{lemma}
        \begin{definition}
          \label{resti.1}
A function fulfilling equation \eqref{resti}, \eqref{riresti} will be
said to be a remainder of order $N$, or simply a remainder.
        \end{definition}

The proof of Lemma \ref{nor.for} consists of a few steps: first we
give a decomposition of an arbitrary symbol in a normal form part, a
nonresonant part and a remainder, then we remove the nonresonant part
of the perturbation and then we iterate. The canonical transformation
used to remove the nonresonant part will be constructed using the Lie
transform method, namely by using the time one flow of an auxiliary
Hamiltonian. This requires the study of the Lie transform in our
$C^\infty$ context. We will also have to solve the cohomological
equation in order to construct the auxiliary Hamiltonian. Finally we
state and prove the iterative Lemma which is the last step of the
proof of the Normal Form Lemma.

\subsection{Cutoffs and splittings}
        
	Let us consider an even $C^\infty$ function $\eta: \R
        \rightarrow \R^+$ such that $\eta(t)  \equiv 1$ if
        $\displaystyle{|t| \leq \frac{1}{2}}$ and $\eta(t) \equiv 0$
        if $\displaystyle{|t| \geq 1\,.}$ For all $k \in \Z^d$ such
        that $0\not= |k| \leq K,$ we define the cut-off function
	\begin{equation}
	\begin{gathered}
		\chi_k(p) = \eta\left( \frac{p \cdot k}{\ep^\delta}\right), \quad p \in \R^d\,,
	\end{gathered}
	\end{equation}
        which is thus supported in $\left|p\cdot k\right|\leq
        \ep ^{\delta}$. 
Consider a smooth family of functions $f_\ep \in\sy m$,
        $f_\ep (p,q)=\sum_{k\in\Z^d}\hat f_k(\ep , p)e^{\im
          k\cdot q}$, we perform for ${f_\ep}$ the following decomposition:
	\begin{equation} \label{dec f}
	f(p ,q) =  f^\res(p, q) + f^\nr(p, q) + f_K(p, q)\,,
	\end{equation}
	where
	\begin{equation}\label{f ris non ris}
	\begin{gathered}
	f^\res(p, q) =  \sum_{0<|k| \leq K} \hat{f}_k(\ep, p)
        \chi_k(p) e^{\im k \cdot q}+\hat f_0(\ep ,p),\\
	f^\nr(p, q) =  \sum_{0<|k| \leq K} \hat{f}_k(\ep, p) \left(1 -\chi_k(p)\right) e^{\im k \cdot q}\,,\\
	f_K (p, q) =  \sum_{|k| >K} \hat{f}_k(\ep, p) e^{\im k \cdot q}\,.
	\end{gathered}
	\end{equation}
        \begin{remark}
          \label{symboli}
If $f\in\sy m$ then $f^{\res},f^{\nr}\in\sy m$. Furthermore $f^{\res}$
is in normal form.  
        \end{remark}
        \begin{lemma}
          \label{restocutoff}
Let $f\in \sy1$, then $f_K\in \sy {1+Na}$, so in particular it is a
remainder in the sense that it fulfills equation \eqref{resti} \eqref{riresti}. 
        \end{lemma}
\proof This is related to the fact that the Fourier coefficients of a
$C^\infty$ function decrease faster than any power of $|k|^{-1}$. Formally we have
to bound the following seminorms
\begin{align}
\nonumber C^{1+Na}_{N_1,N_2}(f_K)=	\sup_{\ep} \sup_{p,|k|> K, |\alpha| = N_1
		}  \left|\frac{\partial_p^{\alpha}
  \hat{f}_k(\ep, p)|k|^{N_2} \ep ^{|\alpha|\delta}}{\ep ^{1+Na}}\right|
\\
\label{sti}
\leq \sup_{\ep} \sup_{p,|k|> K, |\alpha| = N_1
		}  \left|\frac{\partial_p^{\alpha}
  \hat{f}_k(\ep, p)|k|^{N_2+N_3}
  \ep ^{|\alpha|\delta}}{\ep ^{1+Na}K^{N_3}}\right|
\end{align}
and, choosing $N_3>Na/\beta$, one has $K^{N_3}\ep ^{Na}>1$ and thus
$$
\left|\eqref{sti}\right|\leq \sup_{\ep} \sup_{p,|k|> K, |\alpha| = N_1
		}  \left|\frac{\partial_p^{\alpha}
  \hat{f}_k(\ep, p)|k|^{N_2+N_3}
  \ep ^{|\alpha|\delta}}{\ep }\right|=C^{1}_{N_1,N_2+N_3}(f)
$$
which is the thesis.\qed

\subsection{Lie Transform and cohomological equation}\label{Lie}

\begin{definition}
  \label{Lie transform}
Given a function $g\in\sy m$, with $m\geq0$, the time one flow $\Phi^1_g\equiv
\Phi^t_g\big|_{t=1} $ will be called \emph{Lie transform} generated by
$g$. 
\end{definition}
Given a function $f\in\sy {m_1}$, we study $f\circ\Phi^1_g$. To this
end define the sequence $f_{(l)}$, by
\begin{equation}
  \label{seqlie}
f_{(0)}:=f\ ,\quad f_{(l)}:=\poi{f_{(l-1)}}{g}\equiv
\left.\frac{d^l}{dt^l}\right|_{t=0} f\circ\Phi^t_g\ , \quad l\geq1\ ,
\end{equation}
and remark that $f_{(l)}\in\sy {m_1+l(m-\delta)}$ if $g\in\sy m$. 
We have the following lemma.

\begin{lemma}
  \label{resto.Lie}
  Let $g\in\sy{m_2}$ and $f\in\sy{m_1}$, with $m_2\geq 1-\delta$ and
  $m_1\geq 1$ then one has
  \begin{equation}
    \label{lie.tot}
f\circ\Phi^1_g=\sum_{l=0}^N\frac{f_{(l)}}{l!}+\cR^{(N)}_{Lie}\ ,
  \end{equation}
  with $\cR^{(N)}_{Lie}$ a remainder, in the sense that it fulfills
  equation \eqref{resti} and \eqref{riresti}.
\end{lemma}
\proof Use the formula for the remainder of the Taylor series (in
time); this gives
$$
f\circ\Phi^1_g=\sum_{l=0}^N\frac{f_{(l)}}{l!}+ \frac{1}{N!}\int_0^1
(1+s)^Nf_{(N+1)}\circ\Phi^s_gds\ . 
$$
Of course the integral term is $\cR^{(N)}_{Lie} $. To estimate its
supremum it is immediate. To estimate its first differential remark
that
$$
d(f_{(N+1)}\circ\Phi^s_g)=df_{(N+1)}(\Phi^s_g)\circ d \Phi^s_g\ .
$$
Then, from the very definition of the flow one has that its
differential fulfills 
$$
\frac{d}{dt }d \Phi^t_g=dX_g(\Phi^t_g)d \Phi^t_g\ ,
$$
which is estimated by
$$
\norma{\frac{d}{dt }d \Phi^t_g}\leq \ep^{m_2-\delta}\norma {d
  \Phi^t_g}\ ,
$$
where we used the fact that $g$ is a symbol. From this it follows
that, provided $\ep$ is small enough one has $\norma {d
  \Phi^t_g}\leq 2 $ for $|t|\leq 1$.

From this and from the fact that $f_{(l+1)}$ is a symbol the thesis
immediately follows. \qed

Concerning the cohomological equation we have the following simple
lemma

\begin{lemma}
  \label{cohomo}
  Let $f\in \sy m$ and consider the cohomological equation
  \begin{equation}
    \label{cohomol}
\poi{H_0}{g}+f^{\nr}=0\ .
  \end{equation}
  It admits a solution $g\in\sy{m-\delta}$.
\end{lemma}
\proof Expanding in Fourier series, the cohomological equation takes
the form
\begin{align*}
\sum_{j}\im \frac{\partial H_0}{\partial p_j}\sum_{0<|k|\leq K} \im
k_j \hat g_k(p,\ep)e^{\im k\cdot q}=- \sum_{0<|k|\leq K}  \hat
f_k(p,\ep)(1-\chi_k(p,\ep ) )e^{\im k\cdot q}\ ,
\end{align*}
whose solution is 
$$
\hat g_k(p,\ep)=\frac{ \hat
f_k(p,\ep)(1-\chi_k(p,\ep) )}{-\im p\cdot k}\ .
$$
Since $1-\chi_k $ is supported in the region $|p\cdot k|\geq
\ep^\delta/2$, the result follows. \qed 

\subsection{Iterative Lemma}\label{itera}

In this subsection we prove the following lemma.

\begin{lemma}
  \label{interami}
  Let $\ell<N$ be an integer, and let $H^{(\ell)}$ be of the form
  \begin{equation}
    \label{iter.1}
H^{(\ell)}=H_0+Z^{(\ell)}+f_\ell+\cR^{(N)}_{\ell}\ ,
  \end{equation}
  with $Z^{(\ell)}\in\sy1$ in normal form, $f_{\ell}\in\sy{m_\ell}$ with
  $$
m_\ell:=1+\ell a
$$ and $\cR^{(N)}_\ell$ a remainder.

  Then there exists a symbol $g_{\ell+1}\in\sy{m_\ell-\delta}$ which
  generates a Lie transform $\Phi_{g_{\ell+1}}^1$ with the property
  that $H^{(\ell+1)}:=H^{(\ell)}\circ\Phi_{g_{\ell+1}}^1 $ fulfills
  the assumption of the lemma with $\ell+1$ in place of $\ell$.
\end{lemma}
\proof Decompose $f_{\ell}$ as in \eqref{dec f} and let $g_{\ell+1}$ be the
solution of the cohomological equation \eqref{cohomol} with
$f_{\ell}^{\nr}$ in place of $f^{\nr}$ and compute  
\begin{align}
  \label{I.1}
  H^{(\ell)}\circ\Phi_{g_{\ell+1}}^1&= H_0+\poi{H_0}{g_{\ell+1}}
  \\
\label{I.2}
  &+H_0\circ  \Phi_{g_{\ell+1}}^1-\left(H_0+\poi{H_0}{g_{\ell+1}}
\right)
\\
\label{I.3}
&+f_{\ell}^{\nr}+f_{\ell}^{\res}+f_{\ell,K}
\\
\label{I.4}
&+f_\ell\circ  \Phi_{g_{\ell+1}}^1-f_\ell
\\
\label{I.5}
&+Z^{(\ell)}\circ  \Phi_{g_{\ell+1}}^1-Z^{(\ell)}
\\
\label{I.6}
&+Z^{(\ell)}+\cR^{(N)}_\ell\circ\Phi^1_{g_{\ell+1}}\ . 
  \end{align}
Exploiting  Lemma \ref{resto.Lie}, one can decompose the
different lines as 
\begin{align*}
  \eqref{I.4}=\sum_{l=1}^{N}f_{\ell,(l)}+\widetilde\resto^{(N)}_1=:f^1_{\ell+1}+\widetilde\resto^{(N)}_1
  \\
\eqref{I.5}=\sum_{l=1}^{N}Z^{(\ell)}_{(l)}+\widetilde\resto^{(N)}_2=:f^2_{\ell+1}+\widetilde\resto^{(N)}_2  
\end{align*}
with $f^1_{\ell+1}\in\sy {2m_\ell-2\delta}$, $f^2_{\ell+1}\in\sy
{m_\ell+1-2\delta} $ and $\widetilde\cR^{(N)}_j$ remainders (actually
of order higher than $\ep ^{Na+1}$).

Concerning \eqref{I.2}, just remark that the sequence $H_{0,(l)}$
defining the Lie transform of $H_{0}$ (cf. \eqref{seqlie}), can be
generated computing $H_{0,(1)}$ from the cohomological equation,
giving
$H_{0,(1)}=\poi{H_0}{g_{\ell+1}}=-f_{\ell}^{\nr}\in\sy{m_\ell}$. In this
way one gets $H_{0,(l)}\in\sy{2(m_\ell-\delta)}$ and also 
$$
\eqref{I.2}=\sum_{l=2}^{N}H_{0,(l)}+\widetilde\resto^{(N)}_3=:   f^3_{\ell+1}+\widetilde\resto^{(N)}_3\ .  
$$
It follows that, defining
$f_{\ell+1}:=f^1_{\ell+1}+f^2_{\ell+1}+f^3_{\ell+1}$,
$$
\cR^{(N)}_{\ell+1}:=f_{\ell,K}+\widetilde\cR^{(N)}_1+\widetilde\cR^{(N)}_2+\widetilde\cR^{(N)}_3
+\cR^{(N)}_{\ell}\circ \Phi_{g_{\ell+1}}^1\ , \quad
Z^{(\ell+1)}:=Z^{(\ell)}+ f_{\ell}^{\res}\ 
$$
one has the thesis.
\qed

\section{Geometric part}

\subsection{Dynamics of a Normal Form Hamiltonian}

In this section we define a partition of the action space $\R^d$ into
blocks which are left invariant by the flow of a Hamiltonian which is
in normal 
form, namely
\begin{equation}\label{ham fnorm} H_Z(p, q) := H_0(p) + Z(p,
  q)\,,
\end{equation}
with $Z$ in normal form.  As usual this partition will be labeled by
the sub moduli of $\Z^d$ identifying the resonances present in each
region.

\begin{definition}
  \label{modulo}(Module and related notations.)
  A subgroup $M \subseteq \Z^d$ will be called a module if 
$\displaystyle{ \Z^d \cap \textrm{span}_\R M= M}$. Given a Module
  $M$, we will denote $M_{\R}$ the linear subspace of $\R^d$ generated
  by $M$. Furthermore, given a vector $p\in\R^d$ we will denote by
  $p_M$ its orthogonal projection on $M_{\R}$.
  \end{definition}

\begin{remark}
  \label{norm.flow}
 From the Definition \ref{normal form} of normal form it immediately
 follows that, if a point $p\in \R^d$ is such that
$$
|p \cdot k | \geq \ep^\delta \quad \forall k \in \Z^d\backslash \{0\} \textrm{ s.t. } |k| \leq K\,,
$$
then
\begin{equation}\label{non risuona}
\{ p,H_Z\} = 0\,,
\end{equation}
 hence, in this region the action $p$ is conserved along the motion of
 $H_Z$.
\end{remark}

 The first block we define is
\begin{equation}\label{def omega}
E_0 = \left \lbrace p \in \Z^d \left|\ |p \cdot k| \geq \ep^{\delta - 2 \beta} \quad \forall k \textrm{ s.t. } 0 < |k| \leq K \right.\right \rbrace\,,
\end{equation}
where the correction $2\beta$ to the exponent has been inserted in
order to separate $E_0$ from the regions where some resonances are
present.

In order to define the other blocks, we introduce the following
parameters:
\begin{gather} \label{def parametri}
\delta_1 = \delta - 2 \beta,\\
\delta_{s+1} = \delta_s - \beta s \quad \forall s = 1, \dots d-1\,,\\
C_1 = 1\,,\\
C_{s+1} = {3 s 2^s} C_s + 1 \quad \forall s = 1, \dots d-1\,.
\end{gather}
%	where ${\frak v}$ is defined by
%	$$
%	{\frak v} = \min \{ \textrm{Vol}\{u_1| \cdots |u_d\} \ |\ u_1, \dots u_d \textrm{ lin. indep. in } \Z^d \}\,,
%	$$
%	and $\textrm{Vol}\{ u_1| \cdots| u_d\}$ is the volume of the parallelepiped with edges $u_1, \dots u_d\,.$ Recall that ${\frak v}$ is a positive quantity, as pointed out in Lemma \ref{??}. 
The next definition we give is meant to identify the points $p$
which are in resonance with vectors of a given module $M \subseteq
\Z^d\,:$
\begin{definition}[Resonant zones]\label{def zone}
	For any module $M \subseteq \Z^d$ of dimension $s \geq 1$ and
        for any (ordered) set $\{ k_1, \dots, k_s\}$ of linearly
        independent vectors in $M$ such that $ |k_j| \leq K$ for all
        $j = 1, \dots,s,$ we define
	$$
	Z_{k_1, \dots, k_s} = \left \lbrace p \in \Z^d \left|\ |p\cdot k_j| < C_j \ep^{\delta_{j}} \quad \forall j = 1, \dots, d\right.\right \rbrace
	$$
	and
	$$
	Z^{(s)}_M = \bigcup_{\begin{subarray}{c}
		\{k_1, \dots, k_s\} \\ \textrm{lin. ind. in } M
		\end{subarray}} Z_{k_1, \dots, k_s} \,.
	$$
\end{definition}
Remark that the definition of $Z_{k_1,...,k_s}$ depends on the order
in which the vectors $k_j$ are chosen. Thus the definition of
resonant zone slightly differs from the analogous definition of
resonant zone given in \cite{GioPisa}. This is due to the fact that in
the present construction we are interested in exhibiting a
partition, and not only a covering, of the action space $\R^d.$ In
particular we have the following remark.
\begin{remark}\label{zone inscatolate}
	The resonant regions $Z^{(s)}_M$ are not reciprocally
        disjoint; on the contrary, given an arbitrary module $M$ of
        dimension $s\geq 2$, for any $s'<s$, the following inclusion
        holds
	$$
	Z^{(s)}_M \subseteq \bigcup_{M'\, :\,{\rm dim}M'=s' } Z^{(s^\prime)}_{M^\prime}\,.
	$$
\end{remark}
We now define the set composed by the points which are resonant with
the vectors in a module $M,$ but are non-resonant with the vectors
$k\not \in M$.

\begin{definition}[Resonant blocks]\label{def blocchi}
	Let $M $ be a module of dimension $s$, we define the resonant block
	$$
	B^{(s)}_M = Z^{(s)}_M \backslash \left( \bigcup_{\begin{subarray}{c} s^\prime > s \\ \dim M^\prime = s^\prime \end{subarray}} Z^{(s^\prime)}_{M^\prime}\right)\,.
	$$
\end{definition}
We prove below that the resonant blocks $\{B^{(s)}_M\}_{M, s}$ are
reciprocally disjoint; nevertheless, they are not suitable
for the geometric part, since they are not left invariant by the dynamics
associated to a normal form Hamiltonian. For such a reason, we need
the following definition:
\begin{definition}[Extended blocks and fast drift blocks] \label{def blocchi estesi}
For any module $M$  of dimension $s,$ we define
$$ \widetilde{E}^{(s)}_M = \left\{ B^{(s)}_M + M_{\R}\right\} \cap
Z^{(s)}_M
$$ and the \emph{extended blocks}
$$
	E^{(s)}_{M} =  \widetilde{E}^{(s)}_M \backslash \left( \bigcup_{\begin{subarray}{c}
		s^\prime < s \\ \dim M^\prime = s^\prime
		\end{subarray}} E^{(s^\prime)}_{M^\prime}\right)\,,
	$$
	where $A + B$ is the Minkowski sum between sets, namely $\displaystyle{A + B = \left\lbrace a + b\ |\  a \in A, b \in B \right \rbrace}\,.$
	Moreover, for all $p \in E^{(s)}_{M}$ we define the \emph{fast drift blocks}
	$$
	\Pi^{(s)}_{M}(p) = \left\lbrace p + M_{\R} \right \rbrace \cap
        Z^{(s)}_{M}\,\ .
	$$
\end{definition}
With the above definitions, the following result holds true:
\begin{theorem}\label{parte geom}
	\begin{enumerate}
	\item	The blocks $E_0 \cup \{E^{(s)}_M\}_{s, M}$ are a partition of $\R^d\,.$
	\item Each block is invariant for the dynamics of a system in
          normal form.
        \item Along such a dynamics, for any initial datum one has  
	\begin{equation}\label{per sempre}
	|p(t) - p(0)| \leq 3 d 2^{d-1} C_d \ep^{\delta - ((d -1)(d+1)+2)\beta} \,.
	\end{equation}
	\end{enumerate}
\end{theorem}

\begin{corollary}
  \label{dimmain2} Theorem \ref{main2} holds.
  \end{corollary}
\proof Choosing 
{$\delta=\frac{1}{4}+\frac{b}{2}$,}
%$\delta=\frac{1}{4}-\frac{b}{2}$
$\beta=\left(\frac{1}{2}-b\right)\frac{1}{2(d^2+1)}$ and
$N=\left[\frac{M-1}{a}\right]+1$ one immediately gets the result. \qed

The proof of Theorem \ref{parte geom} will occupy the rest of this
subsection.  We start by stating a few geometric results.

\begin{lemma}
  \label{giorgilli}(Lemma 5.7 of \cite{GioPisa})
Let $1\leq s\leq d$, and let $k_1,...,k_s$ be linearly independent
vectors in $\R^d$ satisfying $|k_j|\leq K$ for some positive $K$ and
for $1\leq j\leq s$.  Denote by ${\rm Vol}(k_1,...,k_s)$ the
$s$-dimensional volume of the parallelepiped with sides $k_1,...,k_s$;
let moreover $w\in{\rm Span}(k_1,...,k_s)$ be any vector, and let
$$\alpha:=\max_j \left|w\cdot k_j\right|\ ,
$$
then one has
\begin{equation}
  \label{stigio}
\norma w\leq\frac{sK^{s-1}\alpha}{{\rm Vol}(k_1,...,k_s)}\ .
  \end{equation}
  \end{lemma}
For the proof see \cite{GioPisa}.

\begin{lemma} \label{lemma vicini}(Extended blocks are close to
    blocks.) For all $p \in \widetilde{E}^{(s)}_M$ there exists a
  point $p^\prime \in B^{(s)}_M$ such that
	\begin{equation} \label{eq vicini}
	| p - p^\prime| \leq {2 s} C_s K^{s-1} \ep^{\delta_s}\,.
	\end{equation}
\end{lemma}
\begin{proof}
By the very definition of $\widetilde{E}^{(s)}_M,$ if $p \in
\widetilde{E}^{(s)}_M$ then there exists a point $p^\prime \in
B^{(s)}_M$ such that $\displaystyle{p - p^\prime \in M\,.}$ In
particular one has
that
	\begin{equation} \label{diff}
	p - p^\prime = p_M - p^\prime_M\,.
	\end{equation}
	Moreover, since $p \in \widetilde{E}^{(s)}_M \subseteq
        Z^{(s)}_M,$ there exist $k_1, \dots, k_s$ linearly independent
        vectors in $M$, with $|k_j| \leq K $, 
        such that
	$$
	|p_M \cdot k_j| = |p \cdot k_j| \leq C_j \ep^{\delta_j}\,,
   \quad \forall j = 1, \dots, s\,.
	$$ Hence, remarking that for $s$ independent vectors with
   integer components ${\rm Vol}(k_1,...,k_s)\geq1$, Lemma
   \ref{giorgilli} implies
	$$
	|p_M| \leq {s K^{s-1} C_s \ep^{\delta_s}}\,.
	$$
	Analogously, since $p^\prime \in B^{(s)}_M \subseteq Z^{(s)}_M\,,$ 
	$$
	|p^\prime_M| \leq  {s K^{s-1} C_s \ep^{\delta_s}}\,.
	$$
Thus \eqref{diff} gives 
	$$
	|p-p^\prime| \leq |p_M| + |p^\prime_M| \leq 2 {s K^{s-1} C_s \ep^{\delta_s}}\,.
	$$
\end{proof}
Lemma \ref{lemma vicini} enables us to deduce the following result.

\begin{lemma}\label{lemma lontani}(Non overlapping of resonances)
Consider two arbitrary resonance moduli $M$ and $M^\prime$
respectively of dimensions $s$ and $s^\prime.$ If $s^\prime \leq s$
and $M^\prime \nsubseteq M\,,$ then for all $p \in E^{(s)}_M$ one has
that 	\begin{equation} \label{dist}
	\textrm{dist} \left({\Pi^{(s)}_M}(p),\  Z^{(s^\prime)}_{M^\prime}\right) > {s C_s } K^{s-1} \ep^{\delta_s}\,,
	\end{equation}
	%		$$
	%		\overline{\Pi^{(s)}_M}(p) \cap Z^{(s^\prime)}_{M^\prime} = \emptyset\,,
	%		$$
	where $\displaystyle{\textrm{dist}(A, B) = \inf \{|a-b|\ |\ a \in A,\ b \in B \}}$ denotes the distance between two sets.
\end{lemma}
\begin{proof} By contradiction, assume that \eqref{dist} is not
  true. Then, by Lemma \ref{lemma vicini}, one also has 
  $$
	\textrm{dist} \left(B^{(s)}_M,\  Z^{(s^\prime)}_{M^\prime}\right) \leq {3s C_s }K^{s-1} \ep^{\delta_s}\,,
  $$
It follows that 
  there exist  $p'' \in {B^{(s)}_M}(p)$ and $p^\prime \in Z^{(s^\prime)}_{M^\prime}$ such that
	$$ |{p}'' - p^\prime| \leq 3{ s C_s} K^{s-1}
        \ep^{\delta_s}\,. $$  
	Since ${p}^\prime \in Z^{(s^\prime)}_{M^\prime}$ and $M^\prime \nsubseteq M\,,$ there exists $h\notin M$ such that
	$|h| \leq K$ and $\displaystyle{|p^\prime \cdot h| \leq C_s
          \ep^{\delta_s}\,.}$ Compute now
	\begin{align*}
	|{p}'' \cdot h| & \leq |{p}''-p^\prime| |h| + |{p}^\prime \cdot h| \\
	& \leq 3 {s K^{s-1} C_s} \ep^{\delta_s} K + C_s \ep^{\delta_s}\,,
	\end{align*}
	thus, recalling that $K = [ \ep^{- \beta} ] + 1\,,$ one has that
	$$
	|{p}''\cdot h| < \left({3 s 2^s } + 1\right) C_s \ep^{\delta_s - \beta s}\,.
	$$
	Due to definitions \eqref{def parametri}, it follows that
	$$
	|{p}'' \cdot h | < C_{s+1} \ep^{\delta_{s+1}}\,.
	$$
	Hence ${p}'' \in Z^{(s+1)}_{M_h }\,,$ where $M_h$ is the
        resonance module generated by $M \cup \{h\},$ which is
        impossible, since ${p}'' \in B^{(s)}_M\,$ implies that it is
        not in any higher dimensional resonance zones. 
\end{proof}
\begin{lemma} \label{lemma diam}
	Fix an arbitrary module $M$ of dimension $s$, for all $p \in
        E^{(s)}_M,$
	$$
	\textrm{diam}\ \left( \Pi^{(s)}_M(p) \right) \leq {2 s K^{s-1} C_{s}} \ep^{\delta_{s}}\,.
	$$
\end{lemma}
\begin{proof}
	Arguing as in the proof of Lemma \ref{lemma vicini}, if $p^\prime$ and $p^{\prime \prime}$ are two points in $\Pi^{(s)}_M(p),$ then 
	$$
	|p^\prime_M|,\ |p^{\prime \prime}_M| \leq { s C_s K^{s-1}} \ep^{\delta_s}\,.
	$$
	Hence, recalling that $p^\prime- p^{\prime \prime} \in M,$ we deduce that
	$$
	|p^\prime - p^{\prime \prime}| \leq {2 s C_s K^{s-1}} \ep^{\delta_s}
	$$
\end{proof}
\begin{remark}\label{mini}
	Recall that $K = [\ep^{-\beta}] + 1;$ then Lemma \ref{lemma diam} implies that for any module $M$ and for all $p \in E^{(s)}_M$
	$$
	\textrm{diam} \left(\Pi^{(s)}_M(p)\right) \leq 3 d 2^{d-1} C_d \ep^{\delta - \left((d-1)(d+1) + 2\right)\beta}\,.
	$$
\end{remark}
We are now in position to prove Theorem \ref{parte geom}. Remark that its proof has also as a consequence the fact that, for any resonance modulus $M,$
$$
\partial E^{(s)}_M \subseteq \left(\bigcup_{\begin{subarray}{c} s^\prime < s \\ M^\prime \subset M \end{subarray}} \partial E^{(s^\prime)}_{M^\prime}  \right) \cup \overline{Z^{(s)}_M}\,,
$$
which shows that it is possible to move from the extended block $E^{(s)}_M$ only losing resonances (that is, entering a block $E^{(s^\prime)}_{M^\prime}$ with $M^\prime \subset M$), or  remaining inside the same resonant zone $Z^{(s)}_M\,.$ The latter option will be excluded by the dynamics, which ensures that the motion entirely evolves along planes parallel to $M\,.$
\begin{proof}[Proof of Theorem \ref{parte geom}]
	Since each point $p \in \R^d$ belongs either to $E_0$ or to
        $Z^{(s)}_M$ for some $M$ and $s,$ it immediately follows from
        the definition of the extended blocks that $E_0 \cup
        \{E^{(s)}_{M}\}_{M, s}$ is a covering of $\R^d\,.$ If
        $E^{(s)}_M$ and $E^{(s^{\prime})}_{M^\prime}$ are such that $s
        \neq s^\prime,$ then the two blocks are disjoint by their very
        definition; if $s = s^\prime$ and $M' \neq M,$ then their
        intersection is empty by Lemma \ref{lemma lontani}, recalling
        that for all $M^\prime$ and $s^\prime$ one has
        $E^{(s^{\prime})}_{M^\prime }\subseteq
        Z^{(s^{\prime})}_{M^\prime}.$ This proves Item 1.
        \\
        We now
        prove the invariance of the extended blocks $\{E^{(s)}_{
          M}\}_{s, M}$ along the flow $\Phi^t_{H_Z}\,,$ arguing by
        induction on their dimension $s.$
        \\
 \noindent   {\sc Inductive base: $s
          = 0.$} As already observed in Remark \ref{norm.flow}, if
        $p(0)\in E_0,$ then $p(t) \equiv p(0)\ \forall t
        \in \R\,,$ hence the invariance of the block $E_0$ immediately
        follows.

\noindent        {\sc Inductive step:} Fix $M$ of dimension $s \geq
        1$ and a point $p \in E^{(s)}_M\,.$ We are now going to prove
        that
	\begin{equation}\label{piano}
	\p(t) \in \Pi^{(s)}_M(p)\ , \quad \forall t \in \R\,.
	\end{equation}
	Suppose by contradiction that there exists a finite time $\bar{t}>0$ such that
	$$
	\p(t) \in \Pi^{(s)}_M(p) \quad \forall t < \bar{t}\,, \quad \textrm{ and } \quad  \p(\bar{t}) \notin \Pi^{(s)}_M(p)\,.
	$$
	Then for such $\bar{t}$ one has that 
	\begin{equation}\label{notau}
	{\p(\bar{t}) \in \left \lbrace p + M_\R \right \rbrace\,.}
	\end{equation}
	Indeed, for any normalized vector $\lambda \in \R^d$ such that
        $\lambda \bot M\,,$ consider the quantity $I(t) = \p(t) \cdot
        \lambda\,.$ For all $t$ such that $ 0 \leq t < \bar{t}$ and
        one has
	\begin{align*}
	\dot{I} (t) &= \left \lbrace I (t) ,\ H_Z(p(t), q(t)) \right
        \rbrace \\ & = \sum_{0 < |k| \leq K } i \hat{Z}_k(\ep, p(t))
        (k \cdot \lambda) e^{i k \cdot q(t)} = 0\,,
	\end{align*}
	due to the fact that $\hat{Z}_k(\ep, \p(t)) = 0$ if $k \notin M,$ but $ k \cdot \lambda = 0$ if $k \in M\,.$
	Hence
	$$
	\p(\bar{t}) \cdot \lambda = \underset{t \rightarrow \bar{t}}{\lim}\ \p(t) \cdot \lambda = p \cdot \lambda\,,
	$$ from which \eqref{notau} follows, given the arbitrariety of
        the vector $\lambda \in M^\bot\,.$\\ Recall now the definition
        of $\Pi^{(s)}_M(p)\equiv\{p + M_\R\} \cap Z^{(s)}_M\,;$ since
        by eq. \eqref{notau} $p(\bar{t}) \in \{p + M_\R\}\,,$ it must
        be that
	%$\p(\bar{t}) \notin Z^{(s)}_M\,,$ hence in particular
	\begin{equation}\label{ctrl_dogana}
	p(\bar{t}) \in \partial Z^{(s)}_M\,.
	\end{equation}
	Since $E_0 \cup \{E^{(s)}_M\}_{s, M}$ is a partition of
        $\R^d\,,$ there exists $M^\prime$ such that $p(t) \in
        E^{(s^\prime)}_{M^\prime}$ with $s'=dim M'$ (possibly with $s^\prime = 0,$ if $\p(\bar{t}) \in E_0$). We analyze all the possible configurations.
	\begin{enumerate}
		\item $p(\bar{t}) \in E^{(s^\prime)}_{M^\prime}\,$ with $s^\prime = s\,:$ then, since by its definition $\p(\bar{t}) \notin Z^{(s)}_M\,,$ it must be $M^\prime \neq M\,.$ Thus
		$$
		\p(\bar{t}) \in \partial Z^{(s)}_M \cap Z^{(s)}_{M^\prime}\,,
		$$
		which is empty by Lemma \ref{lemma lontani}. Hence this case is contradictory.
		
		\item $p(\bar{t}) \in E^{(s^\prime)}_{M^\prime}\,$ with $s^\prime > s$. This leads again to a contradiction, since due to Remark \ref{zone inscatolate} one would have
		$$
		p(\bar{t}) \in E^{(s^\prime)}_{M^\prime} \subseteq Z^{(s^\prime)}_{M^\prime} \subseteq \left(\bigcup_{\begin{subarray}{c}
			M^{\prime \prime} \textrm{ of dim. } s  \\ M^{\prime \prime} \neq M
			\end{subarray} } Z^{(s)}_{M^{\prime \prime}} \right)\cup Z^{(s)}_M\,,
		$$
		but $p(\bar{t}) \in \partial Z^{(s)}_M$ implies that neither $\p(\bar{t}) \in Z^{(s)}_M\,,$ nor $\p(\bar{t}) \in Z^{(s)}_{M^{\prime \prime}}$ for any $M^{\prime \prime}$ of dimension $s,$ with $M^{\prime \prime } \neq M,$ due to Lemma \ref{lemma lontani}.
\item $p(\bar{t}) \in E^{(s^\prime)}_{M^\prime}\,$, with $s'<s$. Due to the
  induction assumption, the blocks $E^{(s^\prime)}_{M^\prime}$ of
  dimension $s^\prime < s$ are invariant under the dynamics of
  $H_Z$, thus no orbit can enter or exit from it
	\end{enumerate}
	Hence none of the above situations is possible, contraddicting
        the assumption that $\bar{t}<\infty$. Since the same occurs
        for negative times, we can conclude that
	$$
	p(t) \in \Pi^{(s)}_M(p)  \quad \forall t \in \R\,.
	$$
	Estimate \eqref{per sempre} then follows from Remark \ref{mini}.
	Moreover, since $\displaystyle{\p(t) \in \Pi^{(s)}_M(p) \subseteq \widetilde{E}^{(s)}_M}$ and by inductive hypothesis each block $E^{(s^{\prime})}_{M^\prime}$ with $s^\prime < s$ has been proven to be invariant under the flow of $H_Z$ for all real times, it must be
	$$
	\p(t) \in \widetilde{E}^{(s)}_M \backslash
        \left(\bigcup_{\begin{subarray}{c} s^\prime \leq s \\ \dim
            M^\prime = s^\prime \end{subarray}}
        E^{(s^{\prime})}_{M^\prime} \right)=E^{(s)}_M\ , \quad \forall t \in \R\,.
	$$
\end{proof}

\subsection{Adding the perturbation}
Here we come to study the dynamics of the Hamiltonian $H \circ {\cal
  T}\,.$

In the following, we will denote $(\pp(t), \qq(t)) = \Phi^t_{H \circ
  {\cal T}}(p, q)\,,$ with $H \circ {\cal T}$ as in Lemma
\ref{nor.for}.  Furthermore, in order to be able to study the dynamics
of a point starting in an extended block, say $E^{(s)}_M,$ we consider the following sets:
$$
\left(\Pi^{(s)}_M(p)\right)_{\ep^{\delta} }= \left\lbrace  p^\prime \in \R^d \ |\ \dist(p^\prime, \Pi^{(s)}_M(p)) < \ep^{\delta} \right\rbrace\,.
$$
The result we obtain is the following:
\begin{proposition}\label{coord comode}
	%Let $(\pp(t), \qq(t))$ be the time $t$ flow of the vector field generated by $H \circ {\cal T}\,.$
	For all $N$ there exists a positive threshold $\ep_N$ such that, if $\ep \leq \ep_N,$ then $\forall p \in \R^d$
	\begin{equation}
	|\pp(t) - p| \leq \ep^{\delta - ((d-1)(d+1) + 2) \beta} \quad \forall t \textrm{ s. t. } |t| \leq \ep^{-Na }\,.
	\end{equation}
\end{proposition}
\begin{proof}
Fix $p \in \R^d\,,$ then, for any time $t$ such that $|t| \leq
\ep^{-Na}\,,$ one has that either $\pp(t) \in E_0\,,$ or
$\pp(t)\in\Pi^{(s_t)}_{M_t}(p_t)$, for some fast drift block
identified by a suitable $M_t \subseteq \Z^d$ of dimension $s_t \geq
1$ and some (not unique) $p_t \in E^{(s_t)}_{M_t}$.
Let $t_0 \in [-\ep^{-Na}, \ep^{-Na}]$ be such that $M_{t_0}$ is of minimal
dimension, namely such that 
	$$
	{s_{t_0} = \dim M_{t_0}=\min_{|t|\leq \ep^{-Na}}\dim M_t \,.}
	$$ Of course, if there exists a time $t_0$ such that
        $p(t_0)\in E_0,$ then $M_{t_0} = \{0\}$ and
        $\Pi^{(s_{t_0})}_{M_{t_0}}(p_{t_0})=\left\{p_{t_0}\right\}$
        and
        $\left(\Pi^{(s_{t_0})}_{M_{t_0}}(p_{t_0})\right)_{\ep^\delta}
        \equiv B_{\ep^\delta}(p_{t_0})$, namely the ball of
        center $p_{t_0}$ and radius ${\ep^\delta}$.

        We are going to prove that
	\begin{equation}\label{resto li}
	p(t) \in \left(\Pi^{(s_{t_0})}_{M_{t_0}}(p_{t_0})\right)_{\ep^\delta} \quad \forall |t| \leq \ep^{-N a}\,
	\end{equation}
This is obtained arguing essentially as in the proof of Theorem
\ref{parte geom}.

Let $\bar{t}>0$ be the exit time of $\pp(t)$ from
$\left(\Pi^{(s_{t_0})}_{M_{t_0}}(p_{t_0})\right)_{\ep^\delta}$,
namely the time s.t. $\forall t$ with $0\leq t < \bar{t}$
	$$ \pp(t + t_0) \in
\left(\Pi^{(s_{t_0})}_{M_{t_0}}(p_{t_0})\right)_{\ep^\delta}, \quad
\textrm{ and } \quad \pp(\bar{t} + t_0) \notin
\left(\Pi^{(s_{t_0})}_{M_{t_0}}(p_{t_0}) \right)_{\ep^\delta}\,.
	$$
	We prove that $|\bar{t} + t_0| > \ep^{-N a},$ from which
        \eqref{resto li} follows.
Indeed, suppose by contradiction that $|\bar{t} + t_0| \leq  \ep^{- Na}\,.$ Then for any normalized vector $\lambda \in \R^d$ with $\lambda \bot M_{t_0},$ we consider the quantity
	$$
	I(t) = \pp(t + t_0) \cdot \lambda
	$$
due to Lemma \ref{lemma lontani}, for $0\leq t < \bar{t},$
\begin{align}
  \label{stii}
\left|\dot{I}(t)\right| = \left|\{I(t),\ H \circ {\cal T}\}\right|=\left|\{ I(t),\ R(t)\}\right|\leq K_N \ep^{1 + N a - \delta}
\end{align}
where $K_N$ is a constant bounding the r.h.s. of \eqref{resti}.
Since $|t + t_0| < \ep^{-N a}\,,$ $\displaystyle{|I(t) - I(0)| \leq 2
  K_N \ep^{1 - \delta}\,}$, thus, passing to the limit $t \rightarrow
\bar{t},$ we obtain
$$
|I(\bar{t}) - I(0)| \leq 2 K_N \ep^{1 - \delta}\,,
$$
which is strictly less than $\frac{\ep^\delta}{2}$ if $\ep$ is
small enough.

If $M_{t_0} = \{0\},$ this enables us to conclude that
	$$
	\dist \left(\pp(\bar{t} + t_0) , p_{t_0} \right) < \frac{\ep^\delta}{2}\,,
	$$ which contradicts the definition of $\bar{t}$ as the time
        of exit from $B(p_{t_0}, \ep^\delta)\,;$.
        \\
        Assume now $s_{t_0} \geq 1$, then \eqref{stii} implies
	\begin{equation} \label{vicino al piano}
	\dist \left(\pp(\bar{t} + t_0) , \Pi^{s_{t_0}}_{M_{t_0}}(p_{t_0}) \right) < \frac{\ep^\delta}{2}\,.
	\end{equation}
Since
	$$ \Pi^{(s_{t_0})}_{M_{t_0}}(p_{t_0}) = \{p_{t_0} +
M_{\R}\} \cap Z^{(s_{t_0})}_{M_{t_0}}\,;
$$ now, by the definition of $\bar{t},$ $\pp(\bar{t} + t_0) \notin
\left(\Pi^{(s_{t_0})}_{M_{t_0}}(p_{t_0})\right)_{\ep^\delta}\,$,
equation \eqref{vicino al piano} implies that in particular
$\pp(\bar{t} + t_0) \notin
Z^{(s_{t_0})}_{M_{t_0}}$.
Then one argues as
in the proof of Theorem \ref{parte geom} to deduce that, by Lemma
\ref{lemma lontani}, the point $\pp(\bar{t}+t_0)$ cannot belong to any
block $E^{(s)}_M$ with $s \geq s_{t_0}.$ Thus it must be
	$$
	\pp(\bar{t} + t_0) \in E^{(s)}_{M} \quad \textrm{ for some } s < s_{t_0}\,,
	$$
	which contradicts the minimality hypothesis on $s_{t_0}\,.$
	Hence, arguing analogously for negative times, we can deduce that \eqref{resto li} holds.
	Finally, recall that by Remark \ref{mini} this implies that
	$$
	|\pp(t) - p| \leq 3 d 2^{d-1} C_d \ep^{\delta-\left((d-1)(d+1)+2\right)\beta}\,.
	$$
\end{proof}
Combining the estimate in Proposition \ref{coord comode} and estimate \eqref{restip} on the size of the deformation induced on the action variables by the canonical transformation ${\cal T},$ we are finally able to deduce that for all $N \in \N$ and ${\forall t \in \R}$ such that $|t| \leq \ep^{-N a}\,$ there exists a positive constant $K^\prime_N$ such 
\begin{align*}
|p(t) - p(0)| & \leq |p(t) - \pp(t)| + |\pp(t) - \pp(0)| + |\pp (0) - p(0)|\\
& \leq K^\prime_N \ep^{1 - \delta} + 3 d 2^{d-1} C_d \ep^{\delta - \left((d-1)(d+1) + 2\right) \beta} + K^\prime_N \ep^{1 - \delta}\\
& \leq \left(2 K^\prime_N + 3 d 2^{d-1} C_d \right) \ep^{\delta - \left((d-1)(d+1) + 2\right) \beta}\,,
\end{align*}
which concludes the proof of Theorem \ref{nek}.

 % ---- Bibliography ----
\addcontentsline{toc}{chapter}{Bibliografia}
\bibliographystyle{alpha}
\bibliography{biblio}

\end{document}